\documentclass[a4paper]{amsart}
\usepackage[utf8]{inputenc}
\usepackage[english]{babel}
\usepackage{fancyhdr}
\usepackage{amsmath}
\usepackage{amsfonts,amstext}
\usepackage{amssymb,amsthm,amscd,amsxtra,wasysym,graphicx}
\usepackage{mathrsfs,esint,comment}
\usepackage{tikz-cd}
\usepackage{enumitem,mathtools,dsfont}
\usepackage{palatino,mathpazo}
\usepackage{xcolor}

\def\Bibtex{{\rm B\kern-.05em{\sc i\kern-.025em b}\kern-0.08em T\kern-.1667em\lower.7ex\hbox{E}\kern-.125emX}}
\usepackage[breaklinks=true,bookmarks=false,pagebackref]{hyperref}
\hypersetup{
	unicode=false,          
	pdftoolbar=true,       
	pdfmenubar=true,       
	pdffitwindow=false,     
	pdfstartview={FitH}, 
	pdftitle={Monge--Amp\`ere equations},   
	pdfauthor={Quang-Tuan Dang},   
	colorlinks=true,  
	linkcolor=purple,         
	citecolor=blue,        
	filecolor=green,      
	urlcolor=blue}

\usepackage{cleveref}
\usepackage{indentfirst}
\newtheorem{theorem}{Theorem}
\newtheorem{lemma}[theorem]{Lemma}

\newtheorem{proposition}[theorem]{Proposition}
\theoremstyle{definition}
\newtheorem{definition}[theorem]{Definition}

\numberwithin{theorem}{section}
\numberwithin{equation}{section}
 
\DeclareMathOperator \PSH {{\rm PSH}}

\def\R{\mathbb{R}}

\def\f{\varphi}
\def\dc{dd^c}

\begin{document}
	\title[Continuity of Monge--Amp\`ere potentials]{Continuity of Monge--Amp\`ere Potentials with Prescribed Singularities} 
	\author{Quang-Tuan Dang}

	\address{The Abdus Salam International Centre for Theoretical Physics (ICTP), Str. Costiera, 11, 34151 Trieste, TS, Italy}
	\email{qdang@ictp.it}
	%\thanks{The work is partially supported by the ANR project PARAPLUI}
	\date{\today}
	\begin{abstract}
		We study the continuity
		of solutions to complex Monge–Ampère equations with prescribed singularity type. This generalizes the previous results of DiNezza--Lu (J Reine Angew Math 727:145-167, 2017) and the author~(Int Math Res Not 14:11180-11201, 2022). As an application, we can run the pluripotential Monge--Amp\`ere flows in (J Funct Anal 282(6):65, 2022) starting at a current with prescribed singularities.
	\end{abstract}
	
	\maketitle
	
	%\tableofcontents

	\section{Introduction}
	
	Complex Monge--Amp\`ere equations play a crucial role in studying canonical metrics in K\"ahler geometry, following Yau's solution~\cite{yau1978ricci} to the Calabi conjecture. As evidenced by recent developments in
	connection with the Minimal Model Program, it is thus desirable to construct
	canonical metrics on varieties with mild singularities; see~\cite{eyssidieux2009singular,berman2019kahler} and references therein. They led one to the study of degenerate complex Monge--Amp\`ere equations.
	
	In a series of recent papers~\cite{darvas2018monotonicity,darvas2021log}, Darvas--DiNezza--Lu intensively studied complex Monge--Ampère equations with prescribed singularities. They proved the existence and uniqueness of solutions in the context of big cohomology classes. However, the regularity is unknown.

	To state our main result, let $(X,\omega)$ be a compact K\"ahler manifold of complex dimension $n$ and $\theta$ be a smooth closed (1,1) form representing a big cohomology class. We let $\PSH(X,\theta)$ denote the set of $\theta$-psh functions. Recall that the cohomology class $\{\theta\}$ is {\em big} if there exists $\f\in\PSH(X,\theta-\varepsilon\omega)$ with analytic singularities for some $\varepsilon>0$. Its {\em ample locus} is denoted by $\textrm{Amp}(\theta)$.

	Fixing $\psi\in\PSH(X,\theta)$ and $0\leq f\in L^p(X)$, $p>1$, we look for a solution $\f\in \PSH(X,\theta)$ satisfying
	\begin{equation}\label{eq: CMAE}  
	\theta^n_\f=f\omega^n,\quad
	|\f-\psi|\leq C, 
	\end{equation}
	where $\theta_\f^n$ denotes the non-pluripolar product of $\f$, introduced in~\cite{boucksom2010monge}. The last condition means that $\f$ and $\psi$ have the same singularity type. 
	
	As mentioned in~\cite{darvas2018monotonicity}, for the solvability of~\ref{eq: CMAE}, one needs to impose the necessary condition that $\psi$ has {\em model type singularities}, i.e., $\int_X\theta_\psi^n>0$ and $\psi-P_\theta[\psi]$ is bounded on $X$ (see Sect.~\ref{sect: envelopes} for precise definition).  It is crucial to emphasize that model type singularities
	are natural and appear in many different contexts of complex differential geometry. For instance, all analytic and minimal singularities are of the model type.

	We state the main result, which generalizes the one of DiNezza--Lu~\cite{di2017complex} and of the author~\cite{dang2022continuity}. 
	\begin{theorem}[see Theorem~\ref{thm2}] \label{mainthm}
		Assume that $\chi\in\PSH(X,\theta)$ has analytic singularities. Let $\mu=ge^{-u} dV_X$ be a positive  measure such that $\mu(X)=\int_X\theta_\chi^n>0$, where $g\in L^p(X)$, $p>1$, and $u$ is a quasi-psh function. 
		Then the unique solution $\f\in\mathcal{E}(X,\theta,P_\theta[\chi])$ of the equation
		\[\theta_{\f}^n=\mu,\quad\sup_X\f=0,\]
		is continuous in  $\textrm{Amp}(\theta){\setminus}  (\{\chi=-\infty\}\cup E_{1/q}(u))$, where $q$ is the conjugate exponent of $p$, i.e., $\frac{1}{p}+\frac{1}{q}=1$, and $E_{1/q}(u)=\{x\in X:\nu(u,x)\geq 1/q\}$ with $\nu(u,x)$ being the Lelong number of $u$ at $x$. 
	\end{theorem}
	The existence (and uniqueness) of solution is shown in~\cite{darvas2018monotonicity,darvas2020metric}. The characterization of solutions belonging to weighted subspace was discussed in \cite{do2020complex,darvas2023relative}.
	A celebrated result of Siu shows that $E_{1/q}(u)$ is a closed analytic subset, which means that $\f$ is continuous in a Zariski open set.
	
	\medskip

	The case when $\chi$ has minimal singularities has been shown in ~\cite{dang2022continuity}.  If we additionally assume that $u$ is bounded on $X$, then the solution $\f$ is H\"older continuous in the ample locus $\textrm{Amp}(\theta)$, thanks to~\cite{demailly2014holder}. In case $\mu$ is a smooth volume form, it is expected the solution $\f$ is smooth where $\chi$ is. But even if $\chi$ has minimal singularities, this is a widely open question. According to~\cite{boucksom2010monge}, the answer is affirmative under an extra assumption that $\{\theta\}$ is nef. 
	
	\medskip
	
	The strategy of the proof is the same as that of~\cite{di2017complex,dang2022continuity}. One cannot expect the Monge--Amp\`ere potential $\f$ to be globally bounded in the context of prescribed singularities. In the latter references, the proof relied on a theory of the generalized method of Ko{\l}odziej~\cite{kolodziej1998complex}, which makes use of a theory of generalized Monge--Amp\`ere capacities, developed in~\cite{di2015generalized}, also in~\cite{darvas2018monotonicity,darvas2021log}. We extend the result and propose in this paper an alternative proof by using the quasi-psh envelope technique, recently developed by Guedj--Lu~\cite{guedj2021quasi1,guedj2021quasi}.
	
	\medskip
	
	The result obtained in the context of elliptic equations allows us to have an analogous one in the parabolic case, continuing pluripotential Monge--Amp\`ere flows from degenerate initial data. This generalizes the one in~\cite{dang2022pluripotential}, and
	is briefly discussed in Sect.~\ref{sect: parabolic}.
	\subsubsection*{Acknowledgement} I am  grateful to Tam\'as Darvas for interesting exchanges. Many thanks to Antonio Trusiani for explaining the paper~\cite{trusiani2020continuity}.
	I thank the referee for the useful comments, which improved the presentation of the paper. This work is partially supported by the ANR project PARAPLUI.
	\section{Preliminaries}\label{sect: prel}
	In this section, we recall some terminology and notation.

	\subsection{Non-pluripolar Complex Monge--Ampère Measures}
	We denote $(X,\omega)$ a compact K\"ahler manifold of dimension $n$ and fix $\theta$ a  smooth closed (1,1) form. 
	
	An upper semi-continuous function $\f:X \rightarrow\mathbb{R}\cup\{-\infty\} $
	is called {\em quasi-plurisubharmonic} ({\em quasi-psh} for short) if it is locally the sum of a smooth and a plurisubharmonic (psh for short) function. We say that $\f$ is {\it $\theta$-plurisubharmonic}  ({\it $\theta$-psh} for short) if it is quasi-psh, and $\theta+\dc \f\geq 0$ in the sense of currents, where $d^c$ is normalized so that $\dc=\frac{i}{\pi}\partial\Bar{\partial}$. 
	%By the $\dc$-lemma any closed  positive $(1,1)$-current  $T$ cohomologous to $\theta$ can be written as $T=\theta+\dc\f$ for some $\theta$-psh function $\f$ which is furthermore unique up to an additive constant. 
	We let $\PSH(X,\theta)$ denote the set of all $\theta$-psh functions which are not identically $-\infty$. %This set is endowed with 
	%the weak topology, which coincides  with  the $L^1(X)$-topology. By Hartogs' lemma  $\f\mapsto\sup_X\f$ is continuous in this weak topology. %Since the set of  closed positive currents  in a fixed  cohomology class  is compact (in the weak topology), it follows that the set of $\f\in\PSH(X,\theta)$, with $\sup_X\f=0$ is compact. 
	The cohomology class $\{\theta\}$ is {\em big} if the set $\PSH(X,\theta-\varepsilon\omega)$ is not empty for some $\varepsilon>0$.
	
	From now on, we assume that $\{\theta\}$ is big  unless specified otherwise.
	Following Demailly \cite{demailly1992regularization}, we can find a closed positive $(1,1)$-current $T_0\in \{\theta\}$ such that $$T_0=\theta+\dc\chi\geq 2\delta_0\omega,$$ for some $\delta_0>0$ with $\chi$ a quasi-psh function with \emph{analytic singularities}, i.e., locally %and its global potential $\f$ are said to have \emph{analytic singularities} if there exists $c>0$ such that 
	$
	\chi=c\log\left[\sum_{j=1}^{N}|f_j|^2\right]+v,
	$
	%locally on $X$,	
	where  $v$ is a smooth function and the $f_j$'s are holomorphic functions. We see that on
	\[\Omega:=X\setminus \{\chi=-\infty\},\]
	$\chi$ is smooth.
	It moreover follows from~\cite{boucksom2004divisorial} that we can choose $\chi$ such that $\Omega$ is the largest Zariski open subset of all points $x\in X$ for which there exists a K\"ahler current $T_x\in \alpha$ with analytic singularities such that $T_x$ is smooth in a neighborhood of $x$. This locus is called the {\em ample locus} of $\theta$, also denoted by $\textrm{Amp}(\theta)$. Its complementary is called the non-K\"ahler locus, denoted by $E_{nK}(\theta)$. %there exists a K\"ahler current $T=\theta+\dc\chi$ with analytic singularities
	
	\smallskip
	
	Given $\f,\psi\in \PSH(X,\theta)$, we say that $\f$ is {\it less singular} than  $\psi$, and denote by $\psi\preceq\f$, if  there exists a constant $C$ such that $\psi\leq \f+C$ on $X$. We say that $\f,\psi$ have the {\em same singularity type}, and denote by $\f\simeq\psi$ if $\f\preceq\psi$ and $\psi\preceq \f$. This defines an equivalence relation on
	$\PSH(X,\theta)$, whose equivalence classes are the singularity types $[\f]$. There is a
	natural least singular potential in $\PSH(X,\theta)$ given by
	\begin{align*}
	V_{\theta}:=\sup\{\f\in \PSH(X,\theta): \f\leq 0\}.
	\end{align*}
	A function $\f$ is said to have minimal singularities if it has the same singularity type as $V_\theta$. In particular, $V_\theta=0$ if $\theta$ is semi-positive.
	Note also that $V_\theta$ is locally bounded in the ample locus.
	\smallskip
	
	Let $\theta^1,\cdots,\theta^n$ be closed smooth real (1,1) form representing big cohomology classes and $\f_j\in\PSH(X,\theta^j)$. Following the construction of Bedford--Taylor~\cite{bedford1987fine}, it has been shown in~\cite{boucksom2010monge} that for each $k\in\mathbb{N}$,
	\[ \mathbf{1}_{\cap_j\{\f_j>V_{\theta^j}-k\}}\theta^1_{\max(\f_1,V_{\theta^1}-k)}\wedge\cdots\wedge \theta^n_{\max(\f_n,V_{\theta^n}-k)}\] is well defined as a  Borel positive measure with finite total mass. The sequence of these measures is non-decreasing in $k$ and it converges weakly to the so-called {\em Monge--Amp\`ere product}, denoted by \[ \theta^1_{\f_1}\wedge \cdots\wedge \theta^n_{\f_n}.\] 
	This does not charge pluripolar sets by definition. When $\theta^1=\cdots=\theta^n=\theta$ and $\f_1=\cdots=\f_n$ we obtain the non-pluripolar
	Monge--Amp\`ere measure of $\f$, denoted by $(\theta+\dc\f)^n$ or simply by $\theta_\f^n$.
	
	Let $\phi_j\in \PSH(X,\theta^j)$ be such that $\phi_j$ is less singular than $\f_j$.  By~\cite[Thm. 2.4]{darvas2018monotonicity} we have that \[\int_X\theta^1_{\f_1}\wedge \cdots\wedge \theta^n_{\f_n}\leq \int_X\theta^1_{\phi_1}\wedge \cdots\wedge \theta^n_{\phi_n}. \]
	We say that $\theta^1_{\f_1}\wedge \cdots\wedge \theta^n_{\f_n}$ has {\em full mas} with respect to $\theta^1_{\phi_1}\wedge \cdots\wedge \theta^n_{\phi_n}$ if the equality holds. We let $\mathcal{E}(X,\theta^1_{\phi_1},\ldots,\theta^n_{\phi_n} )$ denote the set of such $n$-tuple $(\f_1,\ldots,\f_n)$. In the particular case when the potentials involved are from the same cohomology
	class $\{\theta\}$, and  with $\phi$ less singular than $\f$ and $\int_X\theta_\f^n=\int_X\theta_{\phi}^n$
	then we simply write $\f\in\mathcal{E}(X,\theta,\phi)$. Also, we simply write $\mathcal{E}(X,\theta)$ when $\phi=V_\theta$. 
	
	We  recall here the plurifine locality  of the non-pluripolar Monge--Amp\`ere measure (see~\cite[Sect. 1.2]{boucksom2010monge}) for later use. 
	\begin{lemma}
		Assume that $\f$, $\psi$ are $\theta$-psh functions such that $\f=\psi$ on an open set $U$ in the plurifine topology. Then 
		\begin{equation*}
		\mathbf{1}_U\theta_\f^n=\mathbf{1}_U\theta_\psi^n.
		\end{equation*}
	\end{lemma}
	For practice, we stress that sets of the form
	$\{u < v\}$, where $u$, $v$ are quasi-psh functions, are open in the plurifine topology.

	We recall
	the following classical inequality (see e.g.,~\cite[Lemma 4.5]{darvas2020metric}).
	\begin{lemma}\label{lem: maxprin}
		Let $\f,\psi\in\PSH(X,\theta)$ be such that $\f\leq \psi$. Then
		\[\mathbf{1}_{\{\f=\psi\}}\theta^n_\f\leq \mathbf{1}_{\{\f=\psi\}}\theta^n_\psi. \]
	\end{lemma}
	\begin{proof} For the reader's convenience, we briefly give proof here. If $\f$ and $\psi$ are locally bounded, the result follows, due to Demailly (see e.g.,~\cite[Theorem 3.23]{guedj2017degenerate}). For general case, set $\f^t:=\max(\f,V_\theta-t)$, $\psi^t:=\max(\psi,V_\theta-t)$. Then $\f^t$ and $\psi^t$ are locally bounded on $\Omega$, it follows that
		\[\mathbf{1}_{\{\f>V_\theta-t\}\cap\{\psi>V_\theta-t \}\cap \{\f^t=\psi^t\}}\theta^n_{\f^t}\leq \mathbf{1}_{\{\f>V_\theta-t\}\cap\{\psi>V_\theta-t \}\cap \{\f^t=\psi^t\}}\theta^n_{\psi^t}, \]
		%multiplying with $\mathbf{1}_{\{\f>V_\theta-t\}\cap\{\psi>V_\theta-t \}}$ both side, and 
		using plurifine locality. Letting $t\to+\infty$, the inequality follows.
		%We refer the reader to~\cite{darvas2020metric}.
	\end{proof}
	\subsection{Quasi-plurisubharmonic Envelopes and Model Potentials}\label{sect: envelopes}
	
	Given a measurable function $h:X\to\mathbb{R}$, we define the {\em $\theta$-psh envelope} of $h$ by
	\begin{equation*}
	P_\theta(h):=(\sup\{u\in\PSH(X,\theta): u\leq h\;\text{on}\, X \})^*,
	\end{equation*} where the star means we take the upper semi-continuous regularization. %We next define 
	%\[ P_{\theta}(g,h):=(\sup\{u\in\PSH(X,\theta): u\leq \min (g,h)\;\text{on}\, X \})^*.\] 
	Given a $\theta$-psh function $\phi$, 
	 Ross and Witt Nystr\"om~\cite{ross2014analytic} introduced the "rooftop envelope"
	\[P_\theta[\phi](h)= \left( \lim_{C\to +\infty}P_\theta(\min(\phi+C,h))\right)^*. \]
	If $h=0$ we simply write $P_\theta[\phi]$. A potential $\phi\in\PSH(X,\theta)$ is called a {\em model potential} if $\int_X\theta_\phi^n>0$ and $\phi=P_\theta[\phi]$.
	
	%We have the following result which has been established in~\cite[Theorem 2.3]{guedj2021quasi2}.
	
	%\begin{theorem}\label{thm: envelope}
	%If $h$ is bounded from below, quasi l.s.c, and  $P_\theta(h)<+\infty$ then
	%\begin{enumerate}
	%   \item $P_\theta(h)$ is a bounded $\theta$-psh function;
	%  \item $P_\theta(h)\leq h$ in $X\setminus P$, for some pluripolar set $P$;
	% \item $(\theta+\dc P_\theta(h))^n$ is concentrated on the contact set $\{P_\theta(h)=h\}$.
	%\end{enumerate}
	%\end{theorem}
	
	\begin{proposition}\label{prop: contact}
		Assume that $h=a\f-b\psi$, where $\f,\psi$ are quasi-psh functions and $a,b$ are positive constants. If $P_\theta(h)\not\equiv -\infty$ then $(\theta+\dc P(h))^n$ is concentrated on the contact set $\{P(h)=h\}$. 
	\end{proposition}
	We note that $h=a\f-b\psi$ is well-defined in the complement of a pluripolar set and by assumption $P_\theta(h)\in\PSH(X,\theta)$. Moreover, $P_\theta(h)\leq a\f-b\psi$ means that $P_\theta(h)+b\psi\leq a\f$ on $X$. A generalized result is proved when $h$ is merely quasi-continuous, c.f.~\cite[Theorem 2.7]{darvas2023relative}.
	\begin{proof}
		It is already well known in some literature~\cite{guedj2019plurisubharmonic,darvas2020metric}. We just sketch the proof here. We assume that $\f$, $\psi$  are $\omega$-psh. Thanks to~\cite{demailly1992regularization}, we can find $\f_j\in \PSH(X,\omega)\cap \mathcal{C}^\infty(X)$ be such that $\f_j\searrow \f$. We set $u_j=P_\theta(a\f_j-b\psi)$ and $u:=P_\theta(h)$, so $u_j\searrow u$ (see~\cite[Prop. 2.2]{guedj2019plurisubharmonic}). Since $a\f_j-b\psi$ is lower semi-continuous, the set $\{u_j<a\f_j-b\psi\}$ is open. It thus follows from a classical balayage argument that for each $j$, $\theta^n_{u_j}$ vanishes on  $\{u_j<a\f_j-b\psi\}$. Also, $u_j\leq a\f_j-b\psi$, hence we have that
		\[\int_X\min (a\f_j-b\psi-u_j,1)\theta_{u_j}^n=0\]
		by~\cite[Prop. 2.5]{guedj2019plurisubharmonic}.
		The functions $\min (a\f_j-b\psi-u_j,1)$ are uniformly bounded, are quasi-continuous, and converge in capacity to $\min (a\f-b\psi-u,1)$ (which is quasi-continuous and bounded on $X$). By letting $j\to+\infty$, the conclusion follows from~\cite[Thm 2.3]{darvas2020metric}.
		%It thus follows from~\cite[Theorem 2.3]{darvas2020metric} that  \[ \]
	\end{proof}

	\begin{proposition}
		\label{prop_imp}
		Let $\f,\psi\in\PSH(X,\theta)$ be such that $\psi$ is more singular than $P_\theta[\f]$. Then for any $b>0$, $P_\omega(b\f-b\psi)$ is a $\omega$-psh function with full Monge--Amp\`ere mass.  
	\end{proposition}
	\begin{proof}
		We adapt the argument in~\cite[Prop. 2.10]{dang2022continuity} which goes back to~\cite{darvas2020metric}.  We assume that $\theta\leq A\omega$ for some $A>0$. For each $j\in\mathbb{N}$ we set $\f_j:=\max(\f,\psi-j)$ and $u_j:=P_\omega(b\f_j-b\psi)$. Then $(u_j)$ is a decreasing sequence of $\omega$-psh functions, and  $u_j\geq -jb$. 
		We will show that $\lim_j u_j$ is not identically $-\infty$. We let for each $j$, $D_j:=\{u_j=b\f_j-b\psi\}$ denote the contact set. Observe that the  $D_j$'s are non-empty for $j$ large enough. Since $u_j+b\psi\leq b\f_j$
		it follows from the maximum principle and Proposition~\ref{prop: contact} that
		\[\omega_{u_j}^n\leq \mathbf{1}_{D_j} [\omega+\dc u_j+b(A\omega+\dc \psi)]^n\leq \mathbf{1}_{D_j}((Ab+1)\omega+\dc b\f_j)^n\]  
		Set $\Tilde{\omega}:=\left( \frac{1}{b}+A\right)\omega$. Fix $t>0$. Since $\f_j=\f$ on $\{\f>V_\theta-t/b\}$ for $j>t/b$, by plurifine locality we have for $j>t/b$, 
		\[\int_{\{\f\leq \psi-t/b\}}\Tilde{\omega}_{u_j}^n=\int_X\Tilde{\omega}_{\f_j}^n-\int_{\{\f>\psi-t/b \}}\Tilde{\omega}_{\f}^n.\]
		We see that  $$\{u_j\leq -t\}\cap D_j=\{\f_j\leq \psi-t/b\}\subset\{\f\leq \psi-t/b\}.$$
		From these things above, we have that
		\begin{equation}\label{eq: ineq1}
		\begin{split}
		\omega_{u_j}^n(u_j\leq -t)&\leq \mathbf{1}_{D_j} b^n\Tilde{\omega}_{\f_j}^n(u_j\leq -t)\\
		&\leq  b^n\Tilde{\omega}_{\f_j}^n(\f\leq \psi-t/b)\\
		&\leq b^n\left(\int_X\Tilde{\omega}_{\f_j}^n-\int_{\{\f>\psi-t/b \}}\Tilde{\omega}_{\f}^n \right).
		\end{split}
		\end{equation}
		Suppose by contradiction that $\sup_X u_j\rightarrow-\infty$ as $j\to+\infty$. It thus follows that $\{u_j\leq -t\}=X$ for $j$ large enough, $t$ being fixed. Hence, for $j>0$ large enough, \eqref{eq: ineq1} becomes
		\begin{equation*}
		\int_X\omega^n\leq b^n\left(\int_X\Tilde{\omega}_{\f_j}^n-\int_{\{\f>\psi-t/b \}}\Tilde{\omega}_{\f}^n \right).
		\end{equation*} %Notice that $\psi_j\geq -jb$, hence $\psi_j$ has full Monge--Amp\`ere mass.   
		Letting $j\rightarrow +\infty$, we obtain
		\begin{equation}\label{ineq: contradict}
		\int_X\omega^n\leq b^n\left(\int_X\Tilde{\omega}_{\f_j}^n-\int_{\{\f>\psi-t/b \}}\Tilde{\omega}_{\f}^n \right),
		\end{equation}
		where we have used that $$\Tilde{\omega}_{\f_j}^n=\sum_{k=0}^n\binom{n}{k}(\Tilde{\omega}-\theta)^k\wedge\theta_{\f_j}^{n-k}\to \Tilde{\omega}_{\f}^n,$$ in the weak sense of measures on $X$, thanks to \cite[Thm. 2.3, Rmk. 2.5]{darvas2018monotonicity}. 
		Indeed, since $\psi$ is more singular than $P_\theta[\f]$ hence $\f\leq \f_j\leq P_\theta[\f]$. By monotonicity, for $k=0,1,\ldots,n$, $$\int_X(\Tilde{\omega}-\theta)^k\wedge\theta_{\f}^{n-k}=\int_X(\Tilde{\omega}-\theta)^k\wedge\theta_{\f_j}^{n-k}=\int_X(\Tilde{\omega}-\theta)^k\wedge\theta_{P_\theta[\f]}^{n-k}.$$ 
		Finally, letting $t\rightarrow+\infty$ in \eqref{ineq: contradict} 
		we obtain a contradiction. Consequently,  $u_j$ decreases to a $\omega$-psh function, we infer that $P_\omega(b\f-b\psi)$ is a $\omega$-psh function for any $b>0$. 
		
		We proceed the same as~\cite[Prop. 2.10]{dang2022continuity} to obtain that $P_\omega(b\f-b\psi)\in \mathcal{E}(X,\omega)$.
	\end{proof}
	
	%\subsection{Demailly's equisingular approximation}
	%We next recall   the basic result on the approximation of quasi-functions by quasi-functions with analytic singularities. We refer the reader to \cite{demailly1992regularization,demailly2015cohomology} and references therein. 

	%Thanks to $\dc$-Lemma, the problem of approximating a positive closed $(1,1)$-current is reduced to approximating a quasi-psh function. 
	%The following result of Demailly \cite{demailly1992regularization,demailly2015cohomology} on the  equisingular approximation  of a quasi-psh function by quasi-psh functions with analytic singularities is crucial:
	
	%\begin{theorem}[Demailly's equisingular approximation]\label{thm: dem} 
	%Let $\f$ be a $\theta$-psh function on $X$. There exists a decreasing sequence of quasi-psh functions $(\f_m)$ such that
	%\begin{enumerate}
	%\item $(\f_m)$ converges pointwise and in $L^1(X)$ to $\f$ as $m\to+\infty$, 
	%\item $\f_m$ has the same singularities as  $1/2m$ times a logarithm of a sum of squares of holomorphic functions,
	% \item $\theta+\dc\f_m\geq -\varepsilon_m\omega$, where $\varepsilon_m>0$ decreases to 0 as $m\to+\infty$,
	%  \item $\int_Xe^{2m(\f_m-\f)}dV<+\infty$;
	%   \item $\f_m$ is smooth outside the analytic subset $E_{1/m}(\f)$.
	%\end{enumerate}
	%\end{theorem}
	%\begin{proof}
	%The proof can be found in \cite[Theorem 1.6, Lemma 1.10]{demailly2015cohomology}. 	
	%\end{proof}

	\section{Proof of the Main Theorem}
	According to~\cite{boucksom2004divisorial} we can find $\chi$  a quasi-psh function with analytic singularities such that \[\theta+\dc\chi\geq 2\delta_0\omega,\] for some $\delta_0>0$ and moreover $\textrm{Amp}(\theta)=X\setminus\{\chi=-\infty\}$. Let $\phi$ be a model type singularity, i.e., $\phi=P_\theta[\phi]$. We assume in this section that $\phi$ is less singular than $\chi$. We emphasize that the assumptions on $\phi$ are truly natural, for instance, $\phi=V_\theta$ or $\phi=P_\theta[\chi]$, as described in~\cite[Remark 1.6]{darvas2018monotonicity}. 
	
	Given a non-negative Radon measure $\mu$ whose total mass is $\int_X\theta^n_\phi>0$,
	we consider the  Monge--Amp\`ere equation
	\begin{equation}\label{cmae}
	\theta^n_\f=\mu.
	\end{equation}
	The systematic study of such equations with prescribed singularities in big cohomology classes has been initiated in \cite{darvas2018monotonicity,darvas2021log}. It has been shown there that~\eqref{cmae} admits a unique normalized solution $\f\in \mathcal{E}(X,\theta,\phi)$ if and only if $\mu$ is {\em non-pluripolar}, i.e., it
	puts no mass on pluripolar subsets.   The characterization of solutions belonging to weighted subspace was proved in \cite{do2020complex,darvas2023relative}.
	
	Our goal  is to prove the following: 
	
	\begin{theorem}\label{thm: mainthm} Assume $\phi$ is as above.
		Let $\f\in\mathcal{E}(X,\theta,\phi)$ be normalized by $\sup_X\f=0$. Assume that $$\theta^n_\f\leq e^{-u}gdV,$$ for some quasi-psh function $u$ on $X$, and $0\leq g\in L^p(dV)$, with $p>1$. 
		Assume that $u$ is locally bounded on an open set $U\subset {\rm Amp}(\theta)$. Then $\f $ is continuous on $U$.
	\end{theorem}
	
	For a quasi-psh function $u$ and $c>0$ we set 
	\[E_c(u):=\{x\in X, \nu(u,x)\geq c\},\] where $\nu(u,x)$ denotes the Lelong number of $u$ at $x$. A well-known result of Siu asserts that the Lelong super-level sets
	$E_c(\f)$ are closed analytic subsets of $X$.
	
	As a consequence of Theorem~\ref{thm: mainthm}  we obtain the following theorem, which implies our theorem in the introductory part.
	\begin{theorem}
		\label{thm2} Assume $\phi$ is as above.
		Assume $\nu=gdV$ to be a Radon measure, with $0\leq g\in L^p(dV)$ for some $p>1$. Let $\mu$ be a non-pluripolar positive measure 
		such that $\mu(X)=\int_X\theta_{\phi}^n$.
		Assume that $\mu=fd\nu$, with $f\leq e^{-u}$ for some quasi-psh function $u$ on $X$. 
		Let $\f\in\mathcal{E}(X,\theta,\phi)$ be the unique normalized solution to \eqref{cmae}. Then
		$\f$ is continuous on $\textrm{Amp}(\theta){\setminus}
		 E_{1/q}(u)$, where $q$ denotes the conjugate exponent of $p$.
	\end{theorem}
	\begin{proof}
		The proof relies on Demailly's equisingular approximation theorem~\cite{demailly2015cohomology}. The same arguments as in~\cite[Theorem 3.1]{dang2022continuity} completes the proof.
	\end{proof}
	\begin{proof}[Proof of Theorem~\ref{thm: mainthm}]
		The proof relies on the quasi-psh envelope technique developed in~\cite{guedj2021quasi1}. This can also be applied to give an alternative proof for~\cite[Theorem 3.1]{dang2022continuity}. 
		We divide the proof into several steps.
		
		\medskip
		\noindent\textbf{Step 1.} {\it We prove that $\f$ is locally bounded on $U$.} 
		
		We pick $a>0$ so that $au$ is $\delta_0\omega$-psh. Set $\psi:=\chi+au$. We thus have $$\theta+\dc\psi\geq \delta_0\omega,$$ and $\psi$ is locally bounded on $U$. We also obtain $\psi\leq \phi+au$. We claim that
		\[\f\geq \psi-A, \]
		for a positive constant $A$ only depending  on $\delta_0$, $p$, $dV$,  and  $\int_Xe^{-P_\omega(a^{-1}\f-a^{-1}\phi)}gdV$. 
		\smallskip 
		
		%We first treat the case that $\f$ is less singular than $\psi$. 
		Set $\Tilde{\f}:=P_{\delta_0\omega}(\f-\psi)$. By Proposition~\ref{prop_imp}, $\Tilde{\f}$ is an $\omega$-psh function since $\psi$ is more singular than $P_\theta[\f]=\phi$ (cf.~\cite[Theorem 1.3]{darvas2018monotonicity}).
		One can show that $\sup_X\Tilde{\f}$ is uniformly bounded from above by applying the argument in~\cite{guedj2017degenerate}.  %We for instance, refer to~\cite[Theorem 3.3]{guedj2021quasi1}.
		Without loss of generality, we can normalize $\sup_X\psi=0$. The set $G:=\{\psi>-1\}$ is nonempty plurifine open set. We observe that $\Tilde{\f}(x)\leq (\f-\psi)(x)\leq 1$ for $x\in G$, hence \[\Tilde{\f}(x)-1\leq V_{G,\omega}:=\sup \{u\in \PSH(X,\omega): u|_G\leq 0\}.\]
		By~\cite[Thm . 9.17.1]{guedj2017degenerate} we have that $\sup_X V_{G,\omega}<+\infty$ since $G$ is non-pluripolar, hence $\sup_X\Tilde{\f}\leq 1+\sup_XV_{G,\omega}<+\infty$. 
		
		Since $\f-\psi$ is bounded from below and quasi-continuous, it follows from  Proposition~\ref{prop: contact} that the Monge--Amp\`ere measure $(\delta_0\omega+\dc\Tilde{\f})^n$ is concentrated on the contact set $D:=\{ \Tilde{\f}+\psi=\f\}$. We thus get
		\[ (\delta_0\omega+\dc\Tilde{\f})^n\leq \mathbf{1}_D(\theta+\dc (\Tilde{\f}+\psi))^n.\]
		We also observe that $\Tilde{\f}+\psi$ is an $\theta$-psh function such that $\Tilde{\f}+\psi\leq \f$ on $X$. It follows from Lemma~\ref{lem: maxprin} that
		\[\mathbf{1}_D(\theta+\dc (\Tilde{\f}+\psi))^n\leq \mathbf{1}_D(\theta+\dc \f)^n. \]
		Therefore, we have
		\begin{equation*}
		\begin{split}
		(\delta_0\omega+\dc\Tilde{\f})^n&\leq   \mathbf{1}_D(\theta+\dc \f)^n\\
		&\leq \mathbf{1}_D e^{-u}dgV\\
		&= \mathbf{1}_D e^{a^{-1}\Tilde{\f}} e^{-a^{-1}(\f-\chi)} gdV\\
		&\leq \mathbf{1}_D e^{a^{-1}\sup_X\Tilde{\f}} e^{-P_\omega(a^{-1}\f-a^{-1}\phi)}   gdV.
		\end{split}
		\end{equation*}
		It follows from the H\"older inequality that $ e^{-P_\omega(a^{-1}\f-a^{-1}\phi)}g$ belongs to $L^{r}(X,dV)$ for some $r\in (1,p)$. By Ko{\l}odziej's estimate~\cite{kolodziej1998complex} (see also~\cite{guedj2021quasi1} for alternative one), we infer that $\Tilde{\f}\geq -A$ is uniformly bounded from below. This proves our claim.
		
		%In general case, we work with $\f^t:=\max(\f,\psi-t)$ instead, for $t>0$ big enough. Note that the contact set $D_t=\{\Tilde{\f}+\psi=\f\}\subset\{\psi-t\leq \f\}$, hence $\mathbf{1}_{D_t}\theta_{\f^t}^n=\mathbf{1}_{D_t}\theta_{\f}^n$ by plurifine locality.
		\medskip

		\noindent\textbf{Step 2.} {\it There exists a sequence of functions $\f_j\in\PSH(X,\theta)\cap \mathcal{C}^0(\textrm{Amp}(\theta))$ such that $\phi\geq \f_j$ decreases to $\f$.}
		
		\smallskip
		
		For convenience, we normalize $\f$ so that $\sup_X\f=-1$. Let $0\geq h_j$ be  a sequence of smooth functions  decreasing to $\f$. Then the sequence of $\theta$-psh functions $\f_j:=P_\theta[\chi](h_j)$ decreases to $\f$ as $j\to +\infty$. 
		Indeed, since the operator $P_\theta$ is monotone, hence the sequence $\f_j$ is decreasing to a $\theta$-psh function $v$. Note that $[\phi]\geq [\f]$ hence $\f$ is a candidate defining $\f_j$ we thus have $\f_j\geq \f$ for all $j$ so $v\geq \f$. Moreover,  $v(x)\leq\f_j(x)\leq h_j(x),\forall\; x\in X$, for all $j$, hence $v(x)\leq \f(x)$,
		as claimed.
		
		Moreover, it follows from~\cite[Theorem 1.1]{McCleerey2020envelopes} that $\f_j$ is continuous outside the singular locus of $\chi$, hence on $\textrm{Amp}(\theta)$.
		%Furthermore, we have that $\f_j$ is  continuous  in $\textrm{Amp}(\alpha)$ for each $j$. In fact, \cite[Theorem 3.8]{darvas2018monotonicity} gives an upper bound on the Monge-Amp\`ere measure of $\f_j$:
		%\begin{equation*}
		%   \theta_{\f_j}^n\leq \mathbf{1}_{\{ \f_j=h_j\}}\theta_{h_j}^n.
		%\end{equation*}
		%The equality also holds following to~\cite[Corollary 3.4.]{di2021monge}. In particular $\theta_{\f_j}^n$ has a bounded density, hence it follows from~\cite[Theorem D]{demailly2014holder} that $\f_j$ is H\"older continuous in $\textrm{Amp}(\theta)$, as claimed.

		\medskip
		
		\noindent\textbf{Step 3.} {\it  Continuity of solutions.} We finally adapt the arguments in~\cite[Theorem 3.7]{guedj2021quasi} to prove the continuity of the solution $\f$ on $U$.  
		\smallskip
		
		%By Demailly's regularization theorem, there exists a sequence of smooth functions $(\f_j)$ such that $\theta+\dc\f_j\geq -2^{-j}\omega$ and $\f_j\searrow \f$. 
		Fix $\lambda\in (0,1)$. Since  Proposition~\ref{prop_imp} ensures that for any $b>0$, $P_{\omega}(b\f-b\phi)$ has zero Lelong numbers everywhere on $X$, 
		it follows from the H\"older inequality that $ge^{-P_\omega(b\f-b\phi)}\in L^r(dV)$ for some $r>1$. It was shown (see e.g.,~\cite{kolodziej1998complex,eyssidieux2009singular}) that there exists a bounded $\delta_0\omega$-psh function solving 
		\[(\delta_0\omega+\dc u_{j,\lambda})^n=e^{u_{j,\lambda}}(g_{j,\lambda}e^{-P_{\omega}(b\f-b\phi)}+h)dV,\] where $h=(\delta_0\omega)^n/dV$, $g_{j,\lambda}=\lambda^{-n}\mathbf{1}_{\{\f<\f_j-\lambda\}}g$. 
		Moreover, $g_{j,\lambda}e^{-P_\omega(b\f-b\phi)}\to 0$ in $L^r$. It thus follows from the stability property (see e.g.,~\cite{guedj2018stability}) that $u_{j,\lambda}$ uniformly converges to 0 as $j\to +\infty$. We consider
		\[v_{j,\lambda}:=(1-\lambda)\f_j +\lambda(\psi+u_{j,\lambda} ) -C\lambda,\]
		for  $C>0$ to be chosen hereafter. We observe that  $v_{j,\lambda}\leq \phi+\lambda au$ since $\chi\leq \phi$ and $\f_j\leq \phi$ by Step 2.
		%Since $\theta+\dc\psi\geq \delta_0\omega$ hence $v_{j,\lambda}$ is $\theta$-psh for $j$ so large that $2^{-j}\leq \frac{\lambda\delta_0}{2(1-\lambda)}$.
		We compute
		\begin{equation}\label{eq: cont1}
		(\theta+\dc v_{j,\lambda})^n\geq e^{u_{j,\lambda}}\mathbf{1}_{\{\f<\f_j-\lambda\}}ge^{-P_\omega(b\f-b\phi)}dV.
		\end{equation}
		By previous step, we have $\f_j\geq \f\geq \psi-A$ for a positive constant $A$. Hence on the set $\{\f<v_{j,\lambda}\}$,
		\[\f<\f_j-\lambda(\f_j-\psi)+\lambda\sup_X|u_{j,\lambda}|-C\lambda\leq \f_j-\lambda, \] where we have chosen $C>1+\sup_X|u_{j,\lambda}|+A$. Moreover on this set we have
		\begin{equation}\label{eq: cont2}
		\begin{split}
		e^{u_{j,\lambda}}ge^{-P_\omega(b\f-b\phi)}&\geq e^{-\sup_X|u_{j,\lambda}|}g e^{-b\f+b\phi} \\
		&\geq e^{-\sup_X|u_{j,\lambda}|}g e^{-bv_{j,\lambda}+b\phi} \\
		&\geq ge^{Ca^{-1}-u},
		\end{split} 
		\end{equation} since $v_{j,\lambda}\leq \phi+\lambda au$
		with $b=(\lambda a)^{-1}$. Therefore, it follows from~\eqref{eq: cont1} and~\eqref{eq: cont2} that
		\[ e^{Ca^{-1}}(\theta+\dc\f)^n\leq (\theta+\dc v_{j,\lambda})^n\leq (\theta+\dc\max(\f,v_{j,\lambda}))^n, \]
		on the set $\{ \f<v_{j,\lambda}\}$. Since $\f\in\mathcal{E}(X,\theta,\phi)$ we can apply the comparison principle (see~\cite[Lemma 2.3]{darvas2020metric}) to get
		\[e^{Ca^{-1}}\int_{\{ \f<v_{j,\lambda}\}}\theta_\f^n\leq \int_{\{ \f<v_{j,\lambda}\}} \theta^n_{\max(\f,v_{j,\lambda})}\leq \int_{\{ \f<v_{j,\lambda}\}}\theta_\f^n.\] It thus follows that $\f\geq \max(\f,v_{j,\lambda})\geq v_{j,\lambda}$ a.e. with respect to the measure $\theta_\f^n$, hence everywhere by the domination principle (see~\cite[Proposition 3.11]{darvas2018monotonicity}). Letting $j\to +\infty$ we obtain
		\[ \liminf_{j\to+\infty} \inf_K(\f-\f_j)\geq -(\sup_K|\psi|+C)\lambda\]
		for any compact set $K\Subset U$. Letting $\lambda\to 0$ we have that $\f_j\to \f$ uniformly on $K$. This completes the proof.
	\end{proof}
	
	\begin{proof}[Proof of Theorem~\ref{mainthm}] We can provide a direct proof of Theorem~\ref{mainthm}, a particular case in Theorem~\ref{thm: mainthm}. Since $\chi$ has analytic singularities there exists a coherent ideal sheaf $\mathcal{J}$ such that on open set $V\subset X$,  $\mathcal{J}$ is generated by holomorphic functions $(f_1,\ldots,f_N)$ and $\chi=c\log\sum|f_j|^2+g$ for $c>0$ and some smooth function $g$ defined locally on $V$. By Hironaka's resolution theorem, 
		there exists 	a composition of blowups with
		smooth centers $\pi:\tilde{X}\to X$ such that $\pi^*\mathcal{J}=\mathcal{O}_{\tilde{X}}(-D)$  for an effective divisor $D=\sum c_jD_j$, $c_j>0$ and by Siu's decomposition, we have
		\[\pi^*(\theta+\dc \chi)=\tilde{\theta}+[D], \] where $\tilde{\theta}$ is a semi-positive closed (1,1) form, $[D]$ denotes the current of integration along $D$. Moreover the map $\pi$ gives a biholomorphic between $\pi^{-1}\textrm{Amp}(\theta)$ and $\textrm{Amp}(\theta)$. So if $\f$ is a $\theta$-psh function of $X$ and satisfies $\f\preceq\chi$ then there exists $\tilde{\f}\in\PSH(\tilde{X},\tilde{\theta})$ such that $\sup_{\tilde{X}}\tilde{u}=\sup_X(u-{P_\theta[\chi]})$ and $\pi^*(\theta_\f)=\tilde{\theta}_{\tilde{\f}}+[D]$. Vice-versa if $\tilde{\f}$ is $\tilde{\theta}$-psh then its restriction on $\pi^{-1}\textrm{Amp}(\theta)$ is of the form $(\f-\chi)\circ \pi$ for some $\theta$-psh function $\f$ on $\textrm{Amp}(\theta)$, such function can be extended to all of $X$.  
		Observe that by definition of the non-pluripolar product, $\pi_*\tilde{\theta}_{\tilde{\f}}^n=\theta_\f^n$ we obtain an 1-1 corresponding between $\mathcal{E}(X,\theta,P_\theta[\chi])$ and $\mathcal{E}(\tilde{X},\tilde{\theta})$. It thus leads to the following equation
		$$\Tilde{\theta}_{\tilde{\f}}^n=e^{-u\circ \pi}\tilde{g}dV_{\tilde{X}},\quad\sup_{\tilde{X}}\tilde{\f}=0,$$
		where $\tilde{g}=\pi^*(gdV_X)/dV_{\tilde{X}}\in L^r$ for some $r>1$ (it has zeros along the exceptional divisor). By the main theorem in \cite{dang2022continuity}, we know that $\tilde{\f}$ is continuous outside $E_{nK}(\tilde{\theta})\cup  E_{1/q}(u\circ \pi)$, where $E_{nK}(\tilde{\theta})$ is the non-K\"ahler locus (cf. \cite[Def. 3.16]{boucksom2004divisorial}). We observe that $$\pi\left( E_{nK}(\tilde{\theta})\cup  E_{1/q}(u\circ \pi)\right)\subset E_{nK}({\theta})\cup \textrm{Sing}(\chi)\cup  E_{1/q}(u),$$ this completes the proof.
	\end{proof}

	\section{Pluripotential Monge--Amp\`ere Flows Through Prescribed Singularities}\label{sect: parabolic}
	%The results obtained in the previous section allow us to obtain analogous ones in the parabolic counterpart.
	We investigate in this section the following complex Monge--Amp\`ere flow
	\begin{align*}\label{cmaf} \tag{CMAF}
	{\rm d}t\wedge(\omega_t+\dc\f_t)^n=e^{\dot{\f}_t+F(t,\cdot,\f_t)}fdV\wedge {\rm d}t
	\end{align*}
	on $X_T$, where 
	\begin{itemize}
		\item $X_T:=(0,T)\times X$ with $T<+\infty$;
		\item $0\leq f\in L^p(X,dV)$ for some $p>1$, and $f>0$ almost everywhere;
		
		\item  $(\omega_t)_{t\in[0,T)}$ is a smooth family of closed $(1,1)$-forms on $X$ such that
		$$g(t)\theta\leq\omega_t,\quad
		\forall\, t\in [0,T),$$ 
		with $g(t)$  an increasing smooth positive function on $[0,T]$.
		
		\item $F:[0,T]\times X\times \R\rightarrow\R$ is  continuous on $[0,T]\times X\times \mathbb{R}$, increasing in $r$ and
		is uniformly Lipschitz, convex in $(t,r)\in[0,T]\times \mathbb{R}$, 
		%\item the function $(t,r)\mapsto F(t,\cdot,r)$ is convex.
		
		\item $\f:[0,T)\times X\to\mathbb{R}$ is the unknown function, with $\f_t=\f(t, \cdot)$.
	\end{itemize}
	
	\smallskip
	
	We consider an $\theta$-psh function $\phi$ as previous, i.e., $\phi$ is a model type and is less singular than $\chi$. 
	We recall that from \cite{darvas2018monotonicity,darvas2021log}
	there exists an $\theta$-psh function $\rho_\phi$ normalized by $\sup_X\rho_\phi=0$  such that
	\[(\theta+\dc\rho_\phi)^n=2^ne^{c_1}fdV,\quad [\rho_\phi]=[\phi], \] where $c_1$ is the normalizing constant such that $\int_X\theta_\phi^n=\int_X 2^ne^{c_1}fdV$. %Moreover, there exists a constant $M>0$ only depending on $\theta$, $dV$, $p$, and $\int_X\theta^n_\phi$ such that 
	%		\begin{align*}
	%		\rho_\phi\geq \phi-M\|f\|_p^{1/n}.
	%		\end{align*}
	
	%\begin{lemma}\label{lem_initial}
	%	For every $\lambda\in [(1+\delta_0g(0))^{-1},1]$,
	%	the function $\lambda g(0)(\rho_\rho+\chi)/2$ is $\omega_0$-psh.
	%\end{lemma}
	%\begin{proof}
	%   The proof is quite the same as that of~\cite[Lemma 1.18]{dang2022pluripotential}.
	%\end{proof}
	Let now $\f_0$ be a $\omega_0$-psh function which is less singular than $ g(0)(\rho_\rho+\chi)/2$. Then there exists a constant $C_0>0$ such that \begin{align*}
	g(0)\frac{\rho_\phi+\chi}{2}-C_0\leq \f_0.
	\end{align*}
	
	%As in~\cite{dang2021pluripotential} we define the space of {\em parabolic potentials}.
	\begin{definition}\label{def_pp}
		The set $\mathcal{P}(X_T,\omega)$  of {\em parabolic potentials} consists of functions $\f: X_T\rightarrow[-\infty,+\infty)$ such that
		\begin{itemize}
			\item $\f$ is upper semi-continuous on $X_T$ and $\f\in L^1_{\rm loc}(X_T)$;
			\item for each $t\in (0,T)$ fixed, the slice $\f_t:x\mapsto\f(t,x)$ is $\omega_t$-psh on $X$;
			\item  
			for any compact subinterval $J\subset(0,T)$, there exists a positive constant $\kappa=\kappa_{J}(\f)$ such that 
			\begin{align}\label{famlip}
			\partial_t \f \leq \kappa-\kappa(\rho_\phi+\chi),
			\end{align}
			in the sense of distributions on $J\times \Omega$.
		\end{itemize}
	\end{definition}

	\begin{definition}
		We say that a parabolic potential $\f\in \mathcal{P}(X_T,\omega)$ is a {\em pluripotential subsolution} to~\eqref{cmaf} on $X_T$ if
		\begin{itemize}
			\item for each $t\in (0,T)$ fixed, the $\omega_t$-psh function $\f(t,\cdot)$ is locally bounded in $\Omega$
			\item the inequality 
			\begin{align*}
			(\omega_t+\dc\f_t)^n\wedge {\rm d}t\geq e^{\dot{\f}_t+F(t,\cdot,\f_t)}fdV\wedge {\rm d}t
			\end{align*}
			holds in the sense of measures in $(0,T)\times \Omega$.
		\end{itemize}
	\end{definition}
	
	\begin{definition}\label{defcp}
		A Cauchy datum for~\eqref{cmaf} is a $\omega_{0}$-psh function $\f_0:X\rightarrow\R$ as above.
		We say $\f \in \mathcal{P}(X_T,\omega)$ is a subsolution to the Cauchy problem:
		\begin{align*}
		(\omega_t+\dc u_t)^n=e^{\partial_{t}u_t+F(t,\cdot,u_t)}fdV,\quad  u\big|_{\{0\}\times X}=\f_0
		\end{align*}
		if $\f$ is a pluripotential subsolution  to \eqref{cmaf} such that
		$\limsup_{t\rightarrow 0}\f(t,x)\leq\f_0(x)$ for all $x\in X$.
		
		We let $\mathcal{S}_{\f_0,f,F}(X_T)$ denote the set of pluripotential subsolutions to the Cauchy problem above.		
	\end{definition}
	
	\begin{lemma}\label{exist_sub}
		The set $\mathcal{S}_{\f_0,f,F}(X_T)$ is non-empty, uniformly bounded from above on $X_T$, and stable under finite maxima.
	\end{lemma}
	\begin{proof}
		We  proceed the same  as in~\cite[Lemma 2.2]{dang2022pluripotential} to obtain a subsolution
		\begin{equation}\label{eq: subsol underline u}
		\underline{u}(t,x):=g(t)\frac{\rho_\phi(x)+\chi(x)}{2}-C_1(t+1),
		\end{equation} and all subsolutions are uniformly bounded by $M_0$.
	\end{proof}
	\begin{definition}
		We let 
		\begin{align*}
		U=U_{\f_0,f,F,X_T}:=\sup\{ \f\in\mathcal{S}_{\f_0,f,F}(X_T): \underline{u}\leq \f\leq M_0 \}
		\end{align*}
		denote the upper envelope of all subsolutions. 
	\end{definition}
	In the same vein, we obtain the regularity in time $t$ of the Perron upper envelope.
	\begin{theorem}\label{thmlips}
		There exists uniform constants $L_U>0$, $C_U>0$ such that for  %$(t,x)\in X_T$,
		\begin{align}\label{lips}
		t|\partial_tU(t,x)|\leq L_U-L_U(\rho_\phi(x)+\chi(x));
		\end{align}
		\begin{align}\label{loc}
		t^2\partial_t^2 U(t,x)\leq C_U-C_U(\rho_\phi(x)+\chi(x)),
		\end{align}
		in the sense of distributions in $X_T$.
	\end{theorem}
	
	\begin{theorem}\label{existence}
		The upper envelope $U:=U_{\f_0,f,F,X_T}$ is a pluripotential solution to the Cauchy problem for the parabolic complex Monge--Amp\`ere equation \eqref{cmaf} in $X_T$. Moreover, $U$ is locally uniformly semi-concave in $(0,T)\times \Omega$.
	\end{theorem}
	\begin{proof}
		The proof relies on a balayage process. We proceed the same as in~\cite[Theorem 3.1]{dang2022pluripotential}.
	\end{proof}
	\begin{theorem}
		\label{thm:  sing}
		% The envelope $U$ is less singular than $\phi$.  
		For each $t\in (0,T)$, $U_t$ is  continuous in $\Omega$.
	\end{theorem}
	\begin{proof}
		We have seen that \[U_t\geq g(t)\frac{\rho_\phi+\chi}{2}-C_1(t+1)\geq g(t)\chi-C(t)\]
		for $C(t)$ a positive constant only depending on $t$. Next, we  proceed as in Step 2, 3 in Theorem~\ref{thm: mainthm}. This completes the proof. 
	\end{proof}
	We can obtain the uniqueness result.
	\begin{theorem}\label{thm_unique} 
		Let $\Phi$ be a pluripotential solution to the Cauchy problem  for \eqref{cmaf} with initial data $\f_0$. 
		Assume that
		\begin{itemize}
			\item $\Phi$ is locally uniformly semi-concave in $(0,T)$;
			\item for each $t$, $\Phi_t$ and $U_t$ have the same singularities;
			\item $\Phi_t$ is continuous in $\Omega$. 
		\end{itemize}
		Then $\Phi=U$.
	\end{theorem}
	\begin{proof}
		We refer the reader to~\cite[Prop. 3.4, Thm. 3.7]{dang2022pluripotential}.
	\end{proof}
	\bibliographystyle{alpha}
	\bibliography{bibfile}	

\newcommand{\etalchar}[1]{$^{#1}$}
\begin{thebibliography}{DDG{\etalchar{+}}14}

\bibitem[BBE{\etalchar{+}}19]{berman2019kahler}
R.~J. Berman, S.~Boucksom, P.~Eyssidieux, V.~Guedj, and A.~Zeriahi.
\newblock K\"{a}hler-{E}instein metrics and the {K}\"{a}hler-{R}icci flow on
  log {F}ano varieties.
\newblock {\em J. Reine Angew. Math.}, 751:27--89, 2019.

\bibitem[BEGZ10]{boucksom2010monge}
S.~Boucksom, P.~Eyssidieux, V.~Guedj, and A.~Zeriahi.
\newblock Monge-{A}mp\`ere equations in big cohomology classes.
\newblock {\em Acta Math.}, 205(2):199--262, 2010.

\bibitem[Bou04]{boucksom2004divisorial}
S.~Boucksom.
\newblock Divisorial {Z}ariski decompositions on compact complex manifolds.
\newblock {\em Ann. Sci. \'{E}cole Norm. Sup. (4)}, 37(1):45--76, 2004.

\bibitem[BT87]{bedford1987fine}
E.~Bedford and B.~A. Taylor.
\newblock Fine topology, \v{S}ilov boundary, and {$(dd^c)^n$}.
\newblock {\em J. Funct. Anal.}, 72(2):225--251, 1987.

\bibitem[Dan22a]{dang2022continuity}
Q.-T. Dang.
\newblock Continuity of {M}onge-{A}mp\`ere potentials in big cohomology
  classes.
\newblock {\em Int. Math. Res. Not. IMRN}, (14):11180--11201, 2022.

\bibitem[Dan22b]{dang2022pluripotential}
Q.-T. Dang.
\newblock Pluripotential {M}onge-{A}mp\`ere flows in big cohomology classes.
\newblock {\em J. Funct. Anal.}, 282(6):Paper No. 109373, 65, 2022.

\bibitem[DDG{\etalchar{+}}14]{demailly2014holder}
J.-P. Demailly, S.~Dinew, V.~Guedj, H.~H. Pham, S.~Ko{\l}odziej, and
  A.~Zeriahi.
\newblock H\"{o}lder continuous solutions to {M}onge-{A}mp\`ere equations.
\newblock {\em J. Eur. Math. Soc. (JEMS)}, 16(4):619--647, 2014.

\bibitem[DDL18]{darvas2018monotonicity}
T.~Darvas, E.~Di{ }Nezza, and C.~H. Lu.
\newblock Monotonicity of nonpluripolar products and complex {M}onge-{A}mp\`ere
  equations with prescribed singularity.
\newblock {\em Anal. PDE}, 11(8):2049--2087, 2018.

\bibitem[DDL21a]{darvas2021log}
T.~Darvas, E.~Di{ }Nezza, and C.~H. Lu.
\newblock Log-concavity of volume and complex {M}onge-{A}mp\`ere equations with
  prescribed singularity.
\newblock {\em Math. Ann.}, 379(1-2):95--132, 2021.

\bibitem[DDL21b]{darvas2020metric}
T.~Darvas, E.~Di{ }Nezza, and C.~H. Lu.
\newblock The metric geometry of singularity types.
\newblock {\em J. Reine Angew. Math.}, 771:137--170, 2021.

\bibitem[DDL23]{darvas2023relative}
T.~Darvas, E.~Di{ }Nezza, and C.~H. Lu.
\newblock {Relative pluripotential theory on compact K\"ahler manifolds}.
\newblock {\em \href{https://arxiv.org/abs/2303.11584}{arXiv:2303.11584}},
  2023.

\bibitem[Dem92]{demailly1992regularization}
J.-P. Demailly.
\newblock Regularization of closed positive currents and intersection theory.
\newblock {\em J. Algebraic Geom.}, 1(3):361--409, 1992.

\bibitem[Dem15]{demailly2015cohomology}
J.-P. Demailly.
\newblock On the cohomology of pseudoeffective line bundles.
\newblock In {\em Complex geometry and dynamics}, volume~10 of {\em Abel
  Symp.}, pages 51--99. Springer, Cham, 2015.

\bibitem[DL15]{di2015generalized}
E.~Di{ }Nezza and C.~H. Lu.
\newblock Generalized {M}onge-{A}mp\`ere capacities.
\newblock {\em Int. Math. Res. Not. IMRN}, (16):7287--7322, 2015.

\bibitem[DL17]{di2017complex}
E.~Di{ }Nezza and C.~H. Lu.
\newblock Complex {M}onge-{A}mp\`ere equations on quasi-projective varieties.
\newblock {\em J. Reine Angew. Math.}, 727:145--167, 2017.

\bibitem[DV22]{do2020complex}
D.~T. Do and D.-V. Vu.
\newblock Complex {M}onge-{A}mp\`ere equations with solutions in finite energy
  classes.
\newblock {\em Math. Res. Lett.}, 29(6):1659--1683, 2022.

\bibitem[EGZ09]{eyssidieux2009singular}
P.~Eyssidieux, V.~Guedj, and A.~Zeriahi.
\newblock Singular {K}\"{a}hler-{E}instein metrics.
\newblock {\em J. Amer. Math. Soc.}, 22(3):607--639, 2009.

\bibitem[GL21]{guedj2021quasi1}
V.~Guedj and C.~H. Lu.
\newblock {Quasi-plurisubharmonic envelopes 1: Uniform estimates on K\"ahler
  manifolds}.
\newblock {\em \href{https://arxiv.org/abs/2106.04273}{arXiv:2106.04273}},
  2021.

\bibitem[GL23]{guedj2021quasi}
V.~Guedj and C.~H. Lu.
\newblock Quasi-plurisubharmonic envelopes 3: {S}olving {M}onge--{A}mp\`ere
  equations on hermitian manifolds.
\newblock {\em J. Reine Angew. Math.}, 800:259--298, 2023.

\bibitem[GLZ18]{guedj2018stability}
V.~Guedj, C.~H. Lu, and A.~Zeriahi.
\newblock Stability of solutions to complex {M}onge-{A}mp\`ere flows.
\newblock {\em Ann. Inst. Fourier (Grenoble)}, 68(7):2819--2836, 2018.

\bibitem[GLZ19]{guedj2019plurisubharmonic}
V.~Guedj, C.~H. Lu, and A.~Zeriahi.
\newblock Plurisubharmonic envelopes and supersolutions.
\newblock {\em J. Differential Geom.}, 113(2):273--313, 2019.

\bibitem[GZ17]{guedj2017degenerate}
V.~Guedj and A.~Zeriahi.
\newblock {\em Degenerate complex {M}onge-{A}mp\`ere equations}, volume~26 of
  {\em EMS Tracts in Mathematics}.
\newblock European Mathematical Society (EMS), Z\"{u}rich, 2017.

\bibitem[Ko{\l}98]{kolodziej1998complex}
S.~Ko{\l}odziej.
\newblock The complex {M}onge-{A}mp\`ere equation.
\newblock {\em Acta Math.}, 180(1):69--117, 1998.

\bibitem[McC20]{McCleerey2020envelopes}
Nicholas McCleerey.
\newblock Envelopes with prescribed singularities.
\newblock {\em J. Geom. Anal.}, 30(4):3716--3741, 2020.

\bibitem[RW14]{ross2014analytic}
J.~Ross and D.~Witt{}Nystr\"{o}m.
\newblock Analytic test configurations and geodesic rays.
\newblock {\em J. Symplectic Geom.}, 12(1):125--169, 2014.

\bibitem[Tru20]{trusiani2020continuity}
A.~Trusiani.
\newblock {Continuity method with movable singularities for classical
  Monge-Amp\`ere equations}.
\newblock {\em \href{https://arxiv.org/abs/2006.09120}{arXiv:2006.09120}, to
  appear in Indiana Univ. Math. J.}, 2020.

\bibitem[Yau78]{yau1978ricci}
S.~T. Yau.
\newblock On the {R}icci curvature of a compact {K}\"{a}hler manifold and the
  complex {M}onge-{A}mp\`ere equation. {I}.
\newblock {\em Comm. Pure Appl. Math.}, 31(3):339--411, 1978.

\end{thebibliography}
	
\end{document}